\documentclass{amsart}

\usepackage{amsmath} 
\usepackage{amsfonts} 
\usepackage{amssymb} 
\usepackage{epsfig} 
\usepackage{graphicx} 
\usepackage{multirow} 
\usepackage{verbatim} 
\usepackage{rotating} 
\usepackage[all]{xy}

\newtheorem{theorem}{Theorem}[section]
\newtheorem{lemma}[theorem]{Lemma}

\theoremstyle{definition}
\newtheorem{definition}[theorem]{Definition}

\theoremstyle{remark}
\newtheorem{remark}[theorem]{Remark}
\newtheorem{notation}[theorem]{Notation}
\numberwithin{equation}{section}
\theoremstyle{corollary}
\newtheorem{corollary}[theorem]{Corollary}

\usepackage[bookmarks=false]{hyperref}

\newcommand{\map}{\mathbf{map}}

\newcommand{\homs}{\mathbf{hom}}
\newcommand{\sSet}{\mathbf{sSet}}

\newcommand{\Sp}{\mathbf{Sp}^{\Sigma}}
\newcommand{\Cat}{\mathbf{Cat}}

\newcommand{\Top}{\mathbf{Top}}

\newcommand{\C}{\mathbf{C}}

\newcommand{\D}{\mathbf{D}}
\newcommand{\A}{\mathbf{A}}
\newcommand{\B}{\mathbf{B}}
\newcommand{\W}{\mathbf{W}}
\newcommand{\M}{\mathbf{M}}

\newcommand{\sCat}{\mathbf{sCat}}

\newcommand{\K}{\mathcal{K}}

\newcommand{\N}{\mathrm{N}_{\bullet}}
\newcommand{\iso}{\mathrm{iso}}
\newcommand{\sing}{\mathrm{sing}}
\newcommand{\Ho}{\mathrm{Ho}}

\newcommand{\colim}{\mathrm{colim}}
\newcommand{\Ob}{\mathrm{Ob}}
\newcommand{\diag}{\mathrm{diag}}

\begin{document}

\title{Grothendiek's Homotopy Hypothesis }

\author{Amrani Ilias}
\address{Department of Mathematics\\ Masaryk University\\ Czech Republic}
\email{ilias.amranifedotov@gmail.com}
\email{amrani@math.muni.cz}


\subjclass[2000]{Primary 55U40, 55P10. Secondary 18F20, 18D25}

\date{October 26, 2011}


\keywords{Model Categories, $\infty$-groupoids}
\maketitle

\begin{abstract}
We construct a "diagonal" cofibrantly generated model structre on the category of simplicial objects in the category of topological categories  $\sCat_{\Top}$, which is the category of diagrams  
$[\mathbf{\Delta}^{op}, \Cat_{\Top}]$. Moreover, we prove that the diagonal model structures is left proper and cellular. We also prove that the category of $\infty-\mathbf{groupoids}$ (the full subcategory of topological categories) has a cofibrantly generated model structure and is Quillen equivalent to the model category of simplicial sets, which proves the Grothendieck's homotopy hypothesis. 

\end{abstract}

\section*{Introduction and Results}

This article can be seen as a first application of the existence of a model structure on the category of small topological categories $\Cat_{\Top}$ \cite{Amrani1}, namely for proving the \textit{Grothendieck's homotopy hypothesis}. Before talking about homotopy hypothesis, we describe our first result related to the algebraic $\K$-theory. In \cite{waldhausen1983}, Waldhausen defined the $\K$-theory of a Waldhausen category $\W$ as homotopy groups of some groupe-like $\mathrm{E}_{\infty}$-space $\K(\W)$. He defined a sort of suspension which takes Waldhausen category $\W$ to a simplicial a Waldhausen category $\mathcal{S}_{\bullet}\W$. This category 
can be considered as a simplicial object in the category of small (topological) categories.
The algebraic $\K$-theory of a suspension $\K(\mathcal{S}_{\bullet}\W)$ is defined as the realization of the nerve taken degree-wise,
more precisely  $\K(\mathcal{S}_{\bullet}\W)= \diag \N w \mathcal{S}_{\bullet}\W$. What is important here
is the interpretation of $ \N w\D$ for a given category $\D$. Indeed, it is the coherent nerve of the (topological)
Dwyer-Kan localization of $w\D$ with respect to $w\D$, i.e., the coherent nerve of the  $\infty$-groupoid  $\mathrm{L}_{w\D}{w\D}:=w\D[w\D^{-1}]$. More precisely, we have a weak equivalence $\N w\D\rightarrow\widetilde{\N} \mathrm{L}_{w\D}{w\D}$ (under some good conditions) . In fact, for each topological category $\A$ we can associate its
underlying $\infty$-groupoid denoted by $\A'$.  
Our idea is to construct a model structure on $\sCat_{\Top}$  \ref{scattop} having the following property:  
$\A_{\bullet}\rightarrow\B_{\bullet}$ is a weak equivalence if and only if $\diag \widetilde{\N} \A_{\bullet}^{'} \rightarrow \diag \widetilde{\N} \B_{\bullet}^{'}$ is a weak equivalence of simplicial sets. 
In \cite{Amrani1}, we have proved that  there is a weak equivalence $k^{!}\widetilde{\N} \A_{\bullet}^{'}\rightarrow \widetilde{\N} \A_{\bullet}$. It means that the left Quillen endofunctor $k^{!}$ capture the  homotopy type of the underlying
$\infty$-groupoid associated to any topological category. Now, we can state our first result as follow\\\\
\textbf{Theorem A.}\ref{scattop}
\textit{There is a cofibrantly generated model structure on $\sCat_{\Top}$ (diagonal model structure) such that    $\A_{\bullet}\rightarrow\B_{\bullet}$ is a weak equivalence (fibration) if and only if $$ \diag~ k^{!}\widetilde{\N} \A_{\bullet}\rightarrow\diag~ k^{!}\widetilde{\N}\B_{\bullet}$$ is a weak equivalence (fibration
) in $\sSet$. Or equivalently, $$ \diag~ \widetilde{\N} \A^{'}_{\bullet}\rightarrow\diag~ \widetilde{\N}\B^{'}_{\bullet}$$  is a weak equivalence in $\sSet$. }\\
In the first section \ref{section1}, we construct a new model structure on 
$\sCat_{\Top}$. In \ref{remarque}, we explain why it is harder to prove the existence of such diagonal model structure on $\sCat_{\sSet}$.
In sections \ref{section2} and \ref{section3}, we prove in details the left properness and the cellularity of the new model structure on $\sCat_{\Top}$. \\
\textbf{Theorem B.}\ref{propgauche}
\textit{The new model structure on $\sCat_{\Top}$ is left proper.} \\
\textbf{Theorem C.}\ref{cellulairescatop}
\textit{The new model structure on $\sCat_{\Top}$ is cellular.} \\

Our goal was to construct the stable model category 
$\Sp(\sCat_{\Top}, S)$ of symmetric spectra over $\sCat_{\Top}$, with respect to some left quillen endofunctor  $S$ (suspension).
Unfortunately the category  $\sCat_{\Top}$ is not simplicial model category, but we believe that combining some technics from \cite{hovey2001} and \cite{dugger2001replacing} we can give an equivalente  model for  $\Sp(\sCat_{\Top}, \mathrm{S})$.\\
Section \ref{section4}, is quite independent from the previous sections. We prove that the category of topological categories which are also 
$\infty$-groupoids is a model category.\\
\textbf{Theorem D.}\ref{modelgroupoid}
\textit{There exists cofibrantly generated model structure on the category of $\infty$-goupoids (definition \ref{defgroup}),
where the weak equivalences are given by 
Dwyer-Kan equivalences.
} \\
Finally, we prove the ultimate theorem related to the \textit{Grothendieck's homotopy hypothesis}\\
\textbf{Theorem (Grothendieck's homotopy hypothesis).}\ref{equivalence}
\textit{The category of infinity groupoids is Quillen equivalent to the category of simplicial sets. 
} \\

\section{model structure}\label{section1}
We will use the same notations as in \cite{Amrani1}
\begin{notation}
~\\
\begin{enumerate}
\item We denote $\Top$ the category of compactly 
generated Hausdorff spaces.
\item $\sSet_{\mathbf{K}}$ denotes the model category on simplicial sets where the fibrant object are
Kan complexes.  $\sSet_{\mathbf{Q}}$ denotes the Joyal model structure where the fibrant objects are quasi-categories ($\infty$-categories).
\item The functor $k_{!}:\sSet_{\mathbf{K}}\rightarrow\sSet_{\mathbf{Q}}$ is defined as the left Kan extension
of the functor with takes $\Delta^{n}$ to the nerve of the groupoid with $n$ objects and only one isomorphism between each 2 objects. Moreover $k_{!}$ has a right adjoint denoted by $k^{!}$.  
\item
The composition of functors 
$$\xymatrix{\sSet_{\mathbf{K}}\ar[r]^{k_{!}}& \sSet_{\mathbf{Q}}\ar[r]^-{\Xi} &\Cat_{\sSet}\ar[r]^{|-|} & \Cat_{\Top}
}$$
is denoted by  $\Theta:\sSet\rightarrow\Cat_{\Top}.$ The composition $\Xi\circ k_{!}$ is denoted by $\widetilde{\Theta}$.
\item The composition 
$$\xymatrix{\Cat_{\Top}\ar[r]^-{\sing}& \Cat_{\sSet}\ar[r]^-{\widetilde{\N}} &\sSet_{\mathbf{Q}}\ar[r]^{k^{!}} & \sSet_{\mathbf{K}}
}$$
is denoted by $\Psi:\Cat_{\Top}\rightarrow\sSet.$  The composition $k^{!}\circ\widetilde{\N}$ is denoted by $\widetilde{\Psi}$.
\item $\sSet^{2}_{d}$ denotes the category of bisimplicial sets provided with the diagonal model structure
called also \textit{Moerdijk model structure}.
There is a Quillen equivalence:
$$\xymatrix{
\sSet_{\mathbf{K}}\ar@<1ex>[r]^-{\mathrm{d}_{\ast}} & \sSet^{2}_{d}\ar@<1ex>[l]^-{\diag} 
}$$
\item $\sSet^{2}_{pr}$ denotes the category of bisimplicial sets provided with the projective model structure. It is well known that every projective weak equivalence is a diagonal equivalence. 
\item The category $\Cat_{\sSet}$ is equipped with Bergner model structure \cite{bergner}, $\Cat_{\Top}$ is equipped with the model structure defined in \cite{Amrani1}. The functors $k_{!}$ \cite{Joyalcrm}, $ ~|-|$ and $\Xi$ are left Quillen funcors. The functors $k^{!}$ \cite{Joyalcrm}, $\sing$ and $\widetilde{\N}$ are right Quillen functors. Moreover, the adjunctions $(\Xi, \widetilde{\N})$ and $(|-|, \sing)$ are Quillen equivalences \cite{lurie}, \cite{Amrani1}.  
\item All objects in $\Cat_{\Top}$ are \textbf{fibrant}. The functor $\sing$ applied to a topological category is a fibrant simplicial category. 

 \end{enumerate}
\end{notation}


We should remind that $(\Theta,~\Psi)$ (resp.   $(\widetilde{\Theta},~\widetilde{\Psi})$)  is a Quillen adjunction because it is a composition of Quillen adjunctions \cite{Amrani1}. This adjoint pair is naturally extended  to an adjunction between  $\sSet^{2}$ and  $\sCat_{\Top}$ (resp. $\sCat_{\sSet}$) denoted by $\Theta_{\bullet},\Psi_{\bullet}$ (resp. $\widetilde{\Theta}_{\bullet},\widetilde{\Psi}_{\bullet}$) . Finally, we define the following adjunction:
$$\xymatrix{
\sSet\ar@<1ex>[r]^-{\mathrm{d}_{\ast}} & \sSet^{2} \ar@<1ex>[l]^-{\diag}\ar@<1ex>[r]^-{\Theta_{\bullet}} &\sCat_{\Top} \ar@<1ex>[l]^-{\Psi_{\bullet}}
}$$ 
Now, we can state the main theorem for this section:

\begin{theorem}[\textbf{A}]\label{scattop}
The adjunction  $(\Theta_{\bullet}\mathrm{d}_{\ast}, \diag\Psi_{\bullet})$ induces a cofibrantly generated model structure on  $\sCat_{\Top}$, where
\begin{enumerate}
\item a morphism $f:\C_{\bullet}\rightarrow \D_{\bullet}$ in $\sCat_{\Top}$  is a weak equivalence (fibration) if 
$$ \diag\Psi_{\bullet}f: \diag\Psi_{\bullet}\C_{\bullet}\rightarrow  \diag\Psi_{\bullet}\D_{\bullet}$$ is a weak equivalence (fibration) in $\sSet_{\mathbf{K}}$,
\item The generating acyclic cofibrations are given by $\Theta_{\bullet}\mathrm{d}_{\ast}\Lambda_{i}^{n}\rightarrow \Theta_{\bullet}\mathrm{d}_{\ast}\Delta^{n}$, for all $0\leq n$ and $0\leq i\leq n$,
\item The generating cofibrations are given by $\Theta_{\bullet}\mathrm{d}_{\ast}\partial\Delta^{n}\rightarrow \Theta_{\bullet}\mathrm{d}_{\ast}\Delta^{n}$, for all $0\leq n$. 
\end{enumerate}
\end{theorem}

We start with a useful lemma which gives us conditions to transfer a model structure by adjunction.
  \begin{lemma}\label{lem1}[\cite{worytkiewicz2007model}, proposition 3.4.1]
Consider an adjunction $$\xymatrix{
\M \ar@<1ex>[r]^-{G} & \C\ar@<1ex>[l]^-{F}
}$$
where $\M$ is a cofibrantly generated model category, with generating cofibrations $\mathrm{I}$ and 
generating trivial cofibrations $\mathrm{J}$. We pose 
\begin{itemize}
\item $\mathrm{W}$ the class of morphisms in $\C$ such the image by  $F$ is a weak equivalence in $\M$.
\item $\mathrm{F}$ the class of morphisms in $\C$ such the image by  $F$ is a fibration in $\M$.
\end{itemize}
We suppose that the following conditions are verified:
\begin{enumerate}
\item The domains  of $G (i)$ are small with respect to  $G (\mathrm{I} )$ for all $i \in \mathrm{I} $ and the domains of  $G(j)$ are small with respect to
$G(\mathrm{J} )$ for all $ j \in \mathrm{J}$.
\item The functor $F$ commutes  with directed colimits i.e.,  
$$F\colim (\lambda \rightarrow \C)= \colim F(\lambda\rightarrow \C).$$
\item Every transfinite composition of weak equivalences in  $\M$ is a weak equivalence.
\item The pushout of $G(j)$ by any morphism $f$ in  $\C$ is in  $\mathrm{W}.$
\end{enumerate}
Then $\C$ forms a model category with weak equivalences  (resp. fibrations) $\mathrm{W}$  (resp. $\mathrm{F})$. Moreover, it is cofibrantly generated with generating cofibrations $G(\mathrm{I})$ and generating trivial cofibrations $G(\mathrm{J})$.
 \end{lemma}

In order to prove the main theorem \ref{scattop} we follow the lemma \ref{lem1}.
\begin{lemma}\label{scattop1}
Let $A$ a simplicial subset of $B$ such that the inclusion $A\rightarrow B$ is a weak equivalence. Let $\C$  an object in $\Cat_{\Top}$. Then for all $ F\in\homs_{\Cat_{\Top}}(\Theta(A),\C)$ the functor $\Psi$ sends the following pushout 
 $$
  \xymatrix{
\Theta (A)\ar[r]^{F}\ar[d]  & \C\ar[d] \\
   \Theta (B)  \ar[r] & \D
  }
  $$
to a homotopy cocartesian square in $\sSet.$
\end{lemma}
\begin{proof}
Since $\Theta$  is a quillen functor, $\Theta (A)\rightarrow \Theta (B)$ is a trivial cofibration in $\Cat_{\Top}$.It implies that $\C\rightarrow\D$ is an equivalence in $\Cat_{\Top}$,
and so $\sing\C\rightarrow \sing\D$ is an equivalence between fibrant objects in $\Cat_{\sSet}$. It follows  that $\widetilde{\N}\sing\C\rightarrow \widetilde{\N}\sing\D$ is an equivalence between fibrant objects  (quasi-category) in
$\sSet_{\mathbf{Q}}$. Finally, $k^{!}\widetilde{\N}\sing\C\rightarrow k^{!}\widetilde{\N}\sing\D$ i.e., $\Psi\C\rightarrow \Psi \D$ is an equivalence in $\sSet_{\mathbf{K}}$.
By the same argument,  $\Psi \Theta(A)\rightarrow \Psi \Theta (B)$ is a weak equivalence in $\sSet_{\mathbf{K}}$. Moreover, it is a monomorphism since $\Theta(A)\rightarrow \Theta (B)$ admits a section (all objects in $\Cat_{\Top}$ are fibrant). So $\Psi \Theta(A)\rightarrow \Psi \Theta (B)$  is a trivial cofibration in $\sSet_{\mathbf{K}}$, consequently $\Psi\C\rightarrow \Psi \Theta(B)\sqcup_{\Psi \Theta(A)}\Psi\C$ is an equivalence in $\sSet_{\mathbf{K}}$. The following diagram summarize the situation:
  $$\xymatrix{
 \Psi \Theta(A)\ar[d]^{\sim} \ar[r]& \Psi\C \ar[d]^{\sim}\ar@/^/[rdd] ^{\sim} \\
    \Psi \Theta(B)\ar[r]  \ar@/_/[rrd] &\Psi \Theta(B)\sqcup_{\Psi \Theta(A)}\Psi\C  \ar@{.>}[rd]^{t} \\
     & &  \Psi\D
  }$$
  
we conclude that  $t: \Psi \Theta(B)\sqcup_{\Psi \Theta(A)}\Psi\C \rightarrow \Psi\D$ is a weak equivalence in  $\sSet_{\mathbf{K}}$
\end{proof}
More generally, we consider the following bisimplicial sets (cf \cite{goerss1999})  
$$B=\mathrm{d}_{\ast}\Delta^{n}=\bigsqcup_{\beta\in\Delta^{n} } \Delta^{n}.$$
$$A=\mathrm{d}_{\ast}\Lambda^{n}_{i}=\bigsqcup_{\beta\in\Lambda^{n}_{i} } C^{\beta}, \mathrm{where~ C^{\beta}~are~weakly~contractible.}$$ 
$$S=\mathrm{d}_{\ast}\partial\Delta^{n}=\bigsqcup_{\beta\in\partial\Delta^{n} } D^{\beta},  \mathrm{where~ D^{\beta}~are~weakly~contractible.}$$

\begin{lemma}\label{scattop2}
If $i:S\rightarrow B$ is a generating cofibration in $\sSet^{2}_{d}$ (resp. an acyclic generating cofibration
  $j:A\rightarrow B$  in $\sSet^{2}_{d}$ ) and $\C_{\bullet}$ an object of $\sCat_{\Top}$, then the functor  $\Psi_{\bullet}$ sends the following pushouts
$$
  \xymatrix{
\Theta_{\bullet} (S)\ar[r]\ar[d]  & \C_{\bullet}\ar[d] & & \Theta_{\bullet} (A)\ar[r]\ar[d]  & \C_{\bullet}\ar[d] \\
   \Theta_{\bullet} (B)  \ar[r] & \D_{\bullet} & &  \Theta_{\bullet} (B)  \ar[r] & \D_{\bullet}
  }
  $$
  to  homotopy cocartesian squares in $\sSet^{2}_{pr}$.
\end{lemma}
\begin{proof}
We will do the proof for  $i:S\rightarrow B$, the other case is analogue.
We denote by $\Delta^{n}(m)$ (resp. $\partial\Delta^{n}(m)$ ) the set of $m$-simplicies  $\Delta^{n}$ (resp. $\partial\Delta^{n}$).\\
First of all, let remark that 
$$j_{m}:S_{m}=\bigsqcup_{\beta\in\partial\Delta^{n}(m) } D^{\beta}\rightarrow\bigsqcup_{\beta\in\partial\Delta^{n}(m)} \Delta^{n}=B^{'}_{m}$$
 is a trivial cofibration in $\sSet_{\mathbf{K}}$.
In an other hand, colimits in $\sCat_{\Top}$ are computed degree-wise. In degree $m$ we have that 
$$\D_{m}=(\C_{m}\bigsqcup_{\Theta S_{m}}\Theta B^{'}_{m})\bigsqcup \bigsqcup_{\beta\in(\Delta^{n}(m)\setminus\partial\Delta^{n}(m))}\Theta(\Delta^{n})  $$ 
If we consider now the pushout in $\sSet^{2}$ 
$$
  \xymatrix{
\Psi_{\bullet}\Theta_{\bullet} (S)\ar[r]\ar[d]  & \Psi_{\bullet}\C_{\bullet}\ar[d] \\
  \Psi_{\bullet} \Theta_{\bullet} (B)  \ar[r] & X
  }$$
then $X_{m}$ is equal to 
 $$(\Psi\C_{m}\bigsqcup_{\Psi\Theta S_{m} }\Psi\Theta  B^{'}_{m}) \bigsqcup \bigsqcup_{\beta\in(\Delta^{n}(m)\setminus\partial\Delta^{n}(m))}\Psi\Theta(\Delta^{n}).  $$
By the lemma \ref{scattop1}, the map $\Psi\C_{m}\bigsqcup_{\Psi\Theta S_{m}}\Psi\Theta B^{'}_{m}\rightarrow \Psi( \C_{m}\bigsqcup_{\Theta S_{m}}\Theta B^{'}_{m})$
is a weak equivalence in $\sSet_{\mathbf{K}}$. Consequently,  $X_{m}\rightarrow \Psi\D_{m}$, is an equivalence for each $m$. So $X\rightarrow \Psi_{\bullet}\D_{\bullet}$ is a weak equivalence in $\sSet^{2}_{pr}$. It follows that a diagonal weak equivalence in $\sSet^{2}_{d}$
\end{proof}

\begin{lemma}\label{scattop3}
Let $A\rightarrow B$ be an acyclic cofibration in $\sSet^{2}_{d}$, then the induced morphisms in $\sSet$,  $\diag\Psi_{\bullet}\Theta_{\bullet} (A)\rightarrow\diag \Psi_{\bullet}\Theta_{\bullet} (B)$, is an acyclic cofibration in $\sSet_{\mathbf{K}}$.
\end{lemma}
\begin{proof}
If $Y\rightarrow \ast$  is an equivalence in $\sSet_{\mathbf{K}}$, then $\Theta(Y)\rightarrow \ast$ is an equivalence in $\Cat_{\Top}$ since $\Theta$ is a left Quillen functor.
We have the following commutative diagram:
$$
\xymatrix{
\Theta_{\bullet} A\ar[r]^{f}\ar[d]  & \bigsqcup_{\Lambda_{i}^{n}}\ast\ar[d] \\
  \Theta_{\bullet} B\ar[r]^{g} & \bigsqcup_{\Delta^{n}}\ast
  }$$
where $f,~g$  are equivalences of topological categories degree by degree. Applying the functor $\Psi_{\bullet}$ we have a degree-wise equivalence of bisimplicial sets $\sSet^{2}_{pr}$, because all objects in  $\Cat_{\Top}$ are fibrant.
Now, applying the diagonal functor, we conclude that  $\diag\Psi_{\bullet}\Theta_{\bullet} (A)\rightarrow\diag \Psi_{\bullet}\Theta_{\bullet} (B)$ is an equivalence.
To see that is in fact a cofibration of simplicial sets, it is enough to see that  $\Theta (C^{\beta})\rightarrow\Theta(\Delta^{n})$ is a trivial cofibration of topological categories, consequently, it admits a section because all objects in $\Cat_{\Top}$ are fibrant.
This implies that $\Psi_{\bullet}\Theta_{\bullet} (A)\rightarrow \Psi_{\bullet}\Theta_{\bullet} (B)$ is a degree-wise monomorphism of bisimplicial sets. Finally, applying the functor  $\diag$ we obtain that 
$$\diag\Psi_{\bullet}\Theta_{\bullet} (A)\rightarrow\diag \Psi_{\bullet}\Theta_{\bullet} (B)$$ 
is a monomorphism in $\sSet$. 
 
\end{proof}

\begin{corollary}\label{scattop4}
With the same notations as in lemma  \ref{scattop2}, the map of bisiplicial sets $ \Psi_{\bullet}\C_{\bullet}\rightarrow X$ is a diagonal weak equivalence. Moreover the map
$\Psi_{\bullet}\C_{\bullet}\rightarrow \Psi_{\bullet}\D_{\bullet}$ is a weak diagonal equivalence.
\end{corollary}
\begin{proof}
Since the functor  $\diag$ commutes with colimits, lemmas \ref{scattop2} and \ref{scattop3}  imply that  $ \diag\Psi_{\bullet}\C_{\bullet}\rightarrow \diag X$ is a weak equivalence. By the lemma \ref{scattop2} we have that 
$\diag X\rightarrow \diag \Psi_{\bullet}\D_{\bullet}$ is a weak equivalence. So the property \textit{2 out of 3} the map  $\Psi_{\bullet}\C_{\bullet}\rightarrow \Psi_{\bullet}\D_{\bullet}$ is a diagonal equivalence.
\end{proof}
\begin{lemma}\label{scattop5}
The functors  $k^{!}$, $\widetilde{\N}$ and $\sing$ commute with directed colimits.
\end{lemma}
\begin{proof}
 The fact  $k^{!}$ commutes with directed colimits is a direct consequence from the adjunction $(k_{!}, k^{!})$ and that the functor $\homs_{\sSet} (k_{!}\Delta^{n},-)$ commutes with directed colimits.
 By the same way $\widetilde{\N}$ commutes with directed colimits since $\Xi (\Delta^{n})$ are small 
 objet in $\Cat_{\sSet}$. 
The functor $\sing:\Cat_{\Top}\rightarrow\Cat_{\sSet}$ commutes with directed colimits by \cite{Amrani1}. 
\end{proof}
Finally, we are ready to prove the main theorem of this section   
\begin{proof} [\textbf{Proof of the main theorem \ref{scattop}.}]
First of all, $\sCat_{\Top}$ is complete and cocomplete because $\Cat_{\Top}$ is so. Following the fundamental lemma \ref{lem1}, the points (1) and (3) are obvious. the point (2) is proven in  \ref{scattop5} and finally, the point (4) is given by \ref{scattop4}.
\end{proof}
\begin{remark}\label{remarque}
We should point out that we are unable to prove a same result for $\sCat_{\sSet}$ for the simple reason that objects in $\Cat_{\sSet}$ are not all fibrant. As we have seen before, it plays a crucial role to
prove the main theorem \ref{scattop}. However, we believe that such model structure exists and is Quillen equivalent to the diagonal model structure on $\sCat_{\Top}$. The main idea is to prove that given any simplicial category $\C$, the counite map $k^{!} \widetilde{\N}\C\rightarrow k^{!} \widetilde{\N}\sing|\C|$ is a weak equivalence in $\sSet_{\mathbf{K}}$, this statement is true if  $\C$ was fibrant.\\

\end{remark}

\section{Left Properness}\label{section2}
In this section we will show that  $\sCat_{\Top}$ is a left proper model category. First of all, we will give some properties of cofibrations.

\begin{lemma}\label{topcof1}
Let $i:A\rightarrow B$ be a generating cofibration in  $\sSet^{2}_{d}$, then $\Theta_{\bullet}i:~\Theta_{\bullet} A\rightarrow \Theta_{\bullet} B$ is an inclusion of topological categories.
Moreover, $\Psi_{\bullet}\Theta_{\bullet}i:~\Psi_{\bullet}\Theta_{\bullet} A\rightarrow \Psi_{\bullet} \Theta_{\bullet} B$ is a monomorphism in $\sSet^{2}$.
\end{lemma}
\begin{proof}
We have seen in \ref{scattop2} that the map $i_{m}:~A_{m}\rightarrow B_{m}$ is written as
$$i_{m}: A_{m}\rightarrow B_{m}^{'}\bigsqcup\bigsqcup_{\beta\in \Delta^{n}(m)\setminus\partial\Delta^{n}(m)} \Delta^{n}.$$
The corestriction map $i^{'}_{m}: A_{m}\rightarrow B_{m}^{'}$ is a trivial cofibration in $\sSet_{\mathbf{K}}$. So, $\Theta i^{'}_{m}: \Theta A_{m}\rightarrow \Theta B_{m}^{'}$ is a trivial cofibration in $\Cat_{\Top}$, consequently, we have a section for $i^{'}_{m}$ because all objects in $\Cat_{\Top}$ are fibrant.
We conclude that  $i_{m}$ is an inclusion of topological categories and $\Psi i_{m}$ is a monomorphism in $\sSet$.
\end{proof}

\begin{lemma}\label{topcof2}
Let $\A_{\bullet}\rightarrow\B_{\bullet}$ be a cellular cofibration obtained by a pushout in $\sCat_{\Top}$ of a generating cofibration  $\Theta_{\bullet}i:\Theta_{\bullet}Z\rightarrow \Theta_{\bullet}W$. Then
$\A_{\bullet}\rightarrow\B_{\bullet}$ is a degree-wise inclusion of topological categories. Moreover, $\Psi_{\bullet}\A_{\bullet}\rightarrow\Psi_{\bullet}\B_{\bullet}$ is a monomorphism in $\sSet^{2}$.
\end{lemma}
\begin{proof}
First of all,
$$\B_{m}=(\A_{m}\bigsqcup_{\Theta Z_{m}}\Theta W^{'}_{m})\bigsqcup_{\beta\in(\Delta^{n}(m)\setminus\partial\Delta^{n}(m))}\Theta(\Delta^{n}),$$
where the corestriction 
$$\Theta_{\bullet}i^{'}_{m}:\A_{m}\rightarrow \A_{m}\bigsqcup_{\Theta Z_{m}}\Theta W^{'}_{m}$$ 
is a trivial cofibration between fibrant objects in $\Cat_{\Top}$. This imply that $\Theta_{\bullet}i^{'}_{m} $ admits a section; it follows that $\Theta_{\bullet}i_{m}$ is a degree-wise inclusion of topological categories  and 
 $$\Psi_{\bullet}\A_{\bullet}\rightarrow\Psi_{\bullet}\B_{\bullet}$$
is a monomorphism in $\sSet^{2}$.
\end{proof}

\begin{corollary}\label{topcof3}
Let $i:\A_{\bullet}\rightarrow \B_{\bullet}$ be a cofibration in $\sCat_{\Top}$, then $i_{m}$ is an inclusion of topological categories and $\Psi_{\bullet}i$ is a degree-wise monomorphism in $\sSet^{2}$.
\end{corollary}
\begin{proof}
For the case of cellular cofibrations, it is a direct consequence of \ref{topcof2}. We know that monomorphisms are colesed under retracts. We conclude that cofibrations in $\sCat_{\Top}$ are degree-wise inclusions. 
 On an other hand, $\Psi =k^{!}\widetilde{\N}\sing$ preserves inclusions, it follows that  $\Psi_{\bullet}i$ is a monomorphism in $\sSet^{2}$.

\end{proof}

\begin{lemma}\label{propgauche2}
Let $\A_{\bullet}\rightarrow\B_{\bullet}$ be a cofibration  obtained by pushout from a generating cofibration  $\Theta_{\bullet}(A)\rightarrow\Theta_{\bullet}(B)$ in $\sCat_{\Top}$. Then
the functor  $\Psi_{\bullet}$ sends the following pushout to 
$$
  \xymatrix{
\A_{\bullet}\ar[r]\ar[d]  & \C_{\bullet}\ar[d] \\
 \B_{\bullet}  \ar[r] & \D_{\bullet}
  }
  $$
  to a homotopy cocartesian squre in $\sSet^{2}_{pr}$.
  More generally, let $\A_{\bullet}\rightarrow\B_{\bullet}$a cellular cofibration in $\sCat_{\Top}$, the we have the same conclusion.
  
  \end{lemma}

\begin{proof}
By the same arguments as in \ref{scattop2}, we have
$$\B_{m}=(\A_{m}\bigsqcup_{\Theta A_{m}}\Theta B^{'}_{m})\bigsqcup\bigsqcup_{\beta\in(\Delta^{n}(m)\setminus\partial\Delta^{n}(m))}\Theta(\Delta^{n})  $$ 
with the property that $\A_{m}\rightarrow \A_{m}\bigsqcup_{\Theta A_{m}}\Theta B^{'}_{m}$ is trivial cofibration in $\Cat_{\Top}$, it follows that it admits a section. Consequently 
 $\Psi\A_{m}\rightarrow \Psi(\A_{m}\bigsqcup_{\Theta A_{m}}\Theta B^{'}_{m})$ is a trivial cofibration in $\sSet$. On the other hand,
$$\D_{m}=\C_{m}\bigsqcup_{\A_{m}} \A_{m}\bigsqcup_{\Theta A_{m}}\Theta B^{'}_{m}\bigsqcup\bigsqcup_{\beta\in(\Delta^{n}(m)\setminus\partial\Delta^{n}(m))}\Theta(\Delta^{n}); $$
applying the functor $\Psi$, we have the universal map in $\sSet$ given by 
$$\Psi\C_{m}\bigsqcup_{\Psi\A_{m}}\Psi \B_{m}\rightarrow \Psi(\C_{m}\bigsqcup_{\A_{m}} \B_{m}).$$
Since $\A_{\bullet}\rightarrow\B_{\bullet}$ is obtained as a pushout of a generating cofibration in $\sCat_{\Top}$, we have

$$\Psi\C_{m}\bigsqcup_{\Psi\A_{m}}\Psi \B_{m}= \Psi\C_{m}\bigsqcup_{\Psi \A_{m}}\Psi(\A_{m}\bigsqcup_{\Theta A_{m}}\Theta B^{'}_{m}) \bigsqcup\bigsqcup_{\beta\in(\Delta^{n}(m)\setminus\partial\Delta^{n}(m))} \Psi\Theta(\Delta^{n})$$
Since $\Psi\A_{m}\rightarrow \Psi(\A_{m}\bigsqcup_{\Theta A_{m}}\Theta B^{'}_{m})$ is a trivial cofibration in $\sSet_{\mathbf{K}}$, then 
$$\Psi\C_{m}\rightarrow  \Psi\C_{m}\bigsqcup_{\Psi \A_{m}}\Psi(\A_{m}\bigsqcup_{\Theta A_{m}}\Theta B^{'}_{m})$$
is also a trivial cofibration.
 On the other hand, 
 $$\Psi\C_{m}\rightarrow  \Psi(\C_{m}\bigsqcup_{\Theta A_{m}}\Theta B^{'}_{m})=\Psi(\C_{m}\bigsqcup_{\A_{m}}\A_{m}\bigsqcup_{\Theta A_{m}}\Theta B^{'}_{m})$$
 is an equivalence by \ref{scattop1}.
 We have the commutative diagram:
 
 $$ \xymatrix{ \Psi \C_{m} \ar[rr]^-{\sim}\ar[rd] ^-{\sim} &&    \Psi\C_{m}\bigsqcup_{\Psi \A_{m}}\Psi(\A_{m}\bigsqcup_{\Theta A_{m}}\Theta B^{'}_{m})  \ar[ld]^-{\sim} \\ &   \Psi(\C_{m}\bigsqcup_{\Theta A_{m}}\Theta B^{'}_{m}).}$$
It follows that 
 $$\xymatrix{
  \Psi\C_{m}\bigsqcup_{\Psi \A_{m}}\Psi(\A_{m}\bigsqcup_{\Theta A_{m}}\Theta B^{'}_{m})\bigsqcup\bigsqcup_{\beta\in(\Delta^{n}(m)\setminus\partial\Delta^{n}(m))}\Psi\Theta(\Delta^{n})\ar[d]\\
   \Psi(\C_{m}\bigsqcup_{\Theta A_{m}}\Theta B^{'}_{m})\bigsqcup\bigsqcup_{\beta\in(\Delta^{n}(m)\setminus\partial\Delta^{n}(m))}\Psi\Theta(\Delta^{n})}$$
 is a weak equivalence in $\sSet_{\mathbf{K}}$.
 consequently
 $$\Psi\C_{m}\bigsqcup_{\Psi\A_{m}}\Psi \B_{m}\rightarrow  \Psi \D_{n}$$
is a weak equivalence.\\
For the general case of cellular cofibrations, we write $i:\A_{\bullet}\rightarrow \B_{\bullet}$ as a transfinite composition
$$\A_{\bullet}\rightarrow\A_{\bullet}^{1}\rightarrow\dots\A_{\bullet}^{\alpha}\rightarrow\A_{\bullet}^{\alpha+1}\rightarrow\dots\rightarrow \B_{\bullet}.$$
We pose  $\C^{\alpha}_{\bullet}=\C_{\bullet}\bigsqcup_{\A_{\bullet}}\A^{\alpha}_{\bullet}$, then the morphism  $\C_{\bullet}\rightarrow\D_{\bullet}$ is a transfinite composition
$$\C_{\bullet}\rightarrow\C_{\bullet}^{1}\rightarrow\dots\C_{\bullet}^{\alpha}\rightarrow\C_{\bullet}^{\alpha+1}\rightarrow\dots\rightarrow \D_{\bullet}$$
 By the precedent case: 
 $$\Psi_{\bullet}\A^{\alpha}_{\bullet}\bigsqcup_{\Psi_{\bullet}\A_{\bullet}}\Psi_{\bullet}\C_{\bullet}\rightarrow \Psi_{\bullet}\C_{\bullet}^{\alpha}$$ 
 is a degree-wise weak equivalence. Moreover, $\Psi_{\bullet}\A^{\alpha}_{\bullet}\rightarrow\Psi_{\bullet}\A^{\alpha+1}_{\bullet}$ is a monomorphism is $\sSet^{2}$ by \ref{topcof3}.
 we conclude that:
 $$\colim_{\alpha}\Psi_{\bullet}\A^{\alpha}_{\bullet}\bigsqcup_{\Psi_{\bullet}\A_{\bullet}}\Psi_{\bullet}\C_{\bullet}\rightarrow \colim_{\alpha}\Psi_{\bullet}\C_{\bullet}^{\alpha}$$ 
is a weak equivalence.
 Noting that  $\Psi_{\bullet}$ commutes with directed colimits, we conclude that 
 $$\Psi_{\bullet}\B_{\bullet}\bigsqcup_{\Psi_{\bullet}\A_{\bullet}}\Psi_{\bullet}\C_{\bullet}\rightarrow\Psi_{\bullet}\D_{\bullet}$$
 is a degree-wise weak equivalence and so a diagonal equivalence. 

\end{proof}

\begin{corollary}\label{propgauche3}
Let $i:\A_{\bullet}\rightarrow\B_{\bullet} $ as in \ref{propgauche2}, the the pushout in $\sCat_{\Top}$
of a weak equivalence along $i$ is a weak equivalence.
\end{corollary}
\begin{proof}
We note the pushout diagram by:
$$
  \xymatrix{
\A_{\bullet}\ar[r]^{\sim}\ar[d]  & \C_{\bullet}\ar[d] \\
 \B_{\bullet}  \ar[r] & \D_{\bullet}.
  }
  $$
 applying the functor $\diag\Psi_{\bullet}$ to the diagram, we obtain a homotopy cocartesian diagram in $\sSet_{pr}^{2}$ . By lemma \ref{topcof3}, the morphism  
  $\Psi_{\bullet}\A_{\bullet}\rightarrow\Psi_{\bullet}\B_{\bullet}$ is a monomorphism in 
  $\sSet^{2}$, consequently $\diag\Psi_{\bullet}\A_{\bullet}\rightarrow\diag\Psi_{\bullet}\B_{\bullet}$ is a cofibration in $\sSet$.
  The following pushout diagram in $\sSet$  summarize the situation:
   $$\xymatrix{
 \diag\Psi_{\bullet} \A_{\bullet}\ar[d] \ar[r]^{\sim}& \diag\Psi_{\bullet}\C_{\bullet} \ar[d]\ar@/^/[rdd]  \\
    \diag\Psi_{\bullet} \B_{\bullet}\ar[r]^{f}  \ar@/_/[rrd]_{t} & X  \ar@{.>}[rd]^{g} \\
     & &  \diag\Psi_{\bullet}\D_{\bullet}
  }$$
Since $\sSet$ is left proper, $f$ is a weak equivalence. Moreover, $g$ is an a weak equivalence by\ref{propgauche2}.consequently, $t$ is a weak equivalence.
\end{proof}

\begin{corollary}\label{propgauche4}
If $i:\A_{\bullet}\rightarrow \B_{\bullet}$ is a cellular cofibration in $\sCat_{\Top}$, then the pushout of a weak equivalence along $i$ is again a weak equivalence.
\end{corollary}
\begin{proof}
Consider the following pushout :
$$
  \xymatrix{
\A_{\bullet}\ar[r]^{\sim}\ar[d]  & \C_{\bullet}\ar[d] \\
 \B_{\bullet}  \ar[r] & \D_{\bullet}.
  }
  $$
We write $i:\A_{\bullet}\rightarrow \B_{\bullet}$ as a transfinite composition of morphisms as described in corollary \ref{propgauche3} i.e.,
$$\A_{\bullet}\rightarrow\A_{\bullet}^{1}\rightarrow\dots\A_{\bullet}^{\alpha}\rightarrow\A_{\bullet}^{\alpha+1}\rightarrow\dots\rightarrow \B_{\bullet}.$$
If we pose $\C^{\alpha}_{\bullet}=\C_{\bullet}\bigsqcup_{\A_{\bullet}}\A^{\alpha}_{\bullet}$, then the morphism $\C_{\bullet}\rightarrow\D_{\bullet}$ is the transfinite composition  
$$\C_{\bullet}\rightarrow\C_{\bullet}^{1}\rightarrow\dots\C_{\bullet}^{\alpha}\rightarrow\C_{\bullet}^{\alpha+1}\rightarrow\dots\rightarrow \D_{\bullet}.$$
By corollary \ref{propgauche3} $\diag\Psi_{\bullet}\A^{\alpha}_{\bullet}\rightarrow \diag\Psi_{\bullet}\C^{\alpha}_{\bullet}$ is a weak equivalence in $\sSet_{\mathbf{K}}$. 
We conclude that 
$$\B_{\bullet}=\colim _{\alpha}\A^{\alpha}_{\bullet}\rightarrow \colim_{\alpha}\C^{\alpha}_{\bullet}=\D_{\bullet}$$
is a weak equivalence in $\sCat_{\Top}$.
\end{proof}

\begin{lemma}\label{propgauche5}
If $i^{'}:\A^{'}_{\bullet}\rightarrow \B^{'}_{\bullet}$ is a retract of a cellular cofibration in $\sCat_{\Top}$, then the  pushout of a weak equivalence along  $i^{'}$ is again a weak equivalence.
\end{lemma}
\begin{proof}
By hypothesis, $i^{'}:\A^{'}_{\bullet}\rightarrow \B^{'}_{\bullet}$ is a retract of some cellular cofibration $i:\A_{\bullet}\rightarrow \B_{\bullet}$. Let the following pushout diagram in $\sCat_{\Top}$ 
$$
  \xymatrix{
\A^{'}_{\bullet}\ar[r]^{\sim}\ar[d]^{i^{'}}  & \C_{\bullet}\ar[d]^{j^{'}} \\
 \B^{'}_{\bullet}  \ar[r] & \B^{'}_{\bullet} \bigsqcup_{\A^{'}_{\bullet} }\C_{\bullet}.
  }
  $$
The retraction between $i$ and $i^{'}$ induces a retraction between  $\C_{\bullet}\rightarrow  \B^{'}_{\bullet} \bigsqcup_{\A^{'}_{\bullet} }\C_{\bullet}=\D^{'}_{\bullet}$ and $\C_{\bullet}\rightarrow\B_{\bullet} \bigsqcup_{\A_{\bullet} }\C_{\bullet}=\D_{\bullet}$.
Consequently,
 $$t^{'}:~ \Psi_{\bullet}\B^{'}_{\bullet} \bigsqcup_{\Psi_{\bullet}\A^{'}_{\bullet} }\Psi_{\bullet}\C_{\bullet} \rightarrow\Psi_{\bullet}\D^{'}_{\bullet}$$ 
 is a retract of 
 $$t:~  \Psi_{\bullet}\B_{\bullet} \bigsqcup_{\Psi_{\bullet}\A_{\bullet} }\Psi_{\bullet}\C_{\bullet}\rightarrow\Psi_{\bullet}\D_{\bullet}.$$ 
 By lemma \ref{propgauche2}, the map $t$ is a weak equivalence and so $t^{'}$ a weak equivalence
  (by retract).
 The map $\diag\Psi_{\bullet}\A_{\bullet}^{'}\rightarrow\diag \Psi_{\bullet}\B_{\bullet}^{'}$ is a cofibration in  $\sSet$ by lemma \ref{topcof3}, so
 $$\Psi_{\bullet}\B_{\bullet}^{'}\rightarrow  \Psi_{\bullet}\B^{'}_{\bullet} \bigsqcup_{\Psi_{\bullet}\A^{'}_{\bullet} }\Psi_{\bullet}\C_{\bullet}$$
 is an weak equivalence (diagonal)  in $\sSet^{2}_{d}$ since $\sSet$ is left proper. Consequently, 
 $$ \Psi_{\bullet}\B_{\bullet}^{'}\rightarrow \Psi_{\bullet}\D_{\bullet}^{'}$$
 is a diagonal equivalence since $t^{'}$ is degree-wise equivalence.
  \end{proof}
\begin{theorem}\label{propgauche}[\textbf{B}]
The model category $\sCat_{\Top}$ is left proper.
\end{theorem}
\begin{proof} It is a direct consequence from \ref{propgauche5}.
\end{proof}


\section{Cellularity of $\sCat_{\Top}$ }\label{section3}
In this section, we prove that  $\sCat_{\Top}$ is a cellular model category (cf \cite{Hirs}).
\begin{lemma}\label{celltop1}
The domains and codomains of generating cofibration of the diagonal model structure on $\sCat_{\Top}$ are compact.
\end{lemma}
\begin{proof}
Suppose that $\C_{\bullet}\rightarrow\D_{\bullet}$ is a cellular cofibration $\sCat_{\Top}$.
Let $\A_{\bullet}\rightarrow\D_{\bullet}$ be a morphism where  $\A_{\bullet}=\Theta_{\bullet}\mathrm{d}_{\ast}X$ is a (co)domain of  some generating cofibration $\sCat_{\Top}$.
The map $\C_{\bullet}\rightarrow\D_{\bullet}$ is written as transfinite composition 
$$ \C_{\bullet}=\C_{\bullet}^{0}\rightarrow\C_{\bullet}^{1}\dots \C_{\bullet}^{s}\rightarrow\dots \D_{\bullet}.$$
Applying the functor $\diag\Psi_{\bullet}$ to this diagram, we obtain:
$$ \diag\Psi_{\bullet}\C_{\bullet}=\diag\Psi_{\bullet}\C_{\bullet}^{0}\rightarrow\diag\Psi_{\bullet}\C_{\bullet}^{1}\dots \rightarrow\diag\Psi_{\bullet}\C_{\bullet}^{s}\rightarrow\dots \diag\Psi_{\bullet}\D_{\bullet}.$$
But $\diag\Psi_{\bullet}\C_{\bullet}^{s}\rightarrow\diag\Psi_{\bullet}\C_{\bullet}^{s+1}$  is a cofibration in  $\sSet$ by \ref{topcof1}.
By adjunction, a map  $\A_{\bullet}\rightarrow\D_{\bullet}$ is the same thing as giving a map $f$ in $\sSet$
$f:X\rightarrow\diag\Psi_{\bullet}\D_{\bullet}$. Since $X$ is compact in $\sSet$, this imply that  $f$ is factored for a certain $s$ by $g:X\rightarrow \diag\Psi_{\bullet}\C_{\bullet}^{s}$.
Using the adjunction again, we conclude that $\A_{\bullet}\rightarrow\D_{\bullet}$ is factored by $\Theta_{\bullet}\mathrm{d}_{\ast}X\rightarrow\C_{\bullet}^{s}$.
\end{proof}

\begin{lemma}\label{celltop2}
The domains of generating acyclic cofibration in $\sCat_{\Top}$ are small relatively to the cellular cofibration.
\end{lemma}
\begin{proof}
We use the same notations as in lemma \ref{celltop1}. Let  $\colim_{s} \C_{\bullet}^{s}$, such that $\C_{\bullet}^{i}\rightarrow\C_{\bullet}^{i+1}$ be a directed colimit  which is a cellular cofibration.
The goal is to prove that $\homs_{\sCat_{\Top}}(\A_{\bullet},-)$ commutes with directed colimits, where $\A_{\bullet}=\Theta_{\bullet}\mathrm{d}_{\ast}X$ is a domain of an acyclic cofibration  in $\sCat_{\Top}$.
Again, by adjunction  we have,
$$\homs_{\sCat_{\Top}}(\A_{\bullet},\colim_{s} \C_{\bullet}^{s})=\homs_{\sSet}(X,\diag\Phi_{\bullet}\colim_{s}\C_{\bullet}^{s}).$$
But $\diag\Phi_{\bullet}$ commutes with directed colimits, so 
$$\homs_{\sSet}(X,\diag\Phi_{\bullet}\colim_{s}\C_{\bullet}^{s})=\homs_{\sSet}(X,\colim_{s}\diag\Phi_{\bullet}\C_{\bullet}^{s}).$$
But all objects in $\sSet$ are small. Consequently:
$$\homs_{\sSet}(X,\colim_{s}\diag\Phi_{\bullet}\C_{\bullet}^{s})=\colim_{s}\homs_{\sSet}(X,\diag\Phi_{\bullet}\C_{\bullet}^{s}).$$
Finally, we conclude by adjunction that
$$\homs_{\sCat_{\Top}}(\A_{\bullet},\colim_{s} \C_{\bullet}^{s})=\colim_{s}\homs_{\sCat_{\Top}}(\A_{\bullet}, \C_{\bullet}^{s}).$$
\end{proof}
\begin{lemma}\label{celltop3}
The cofibration in $\sCat_{\Top}$ are effective monomorphisms.
\end{lemma}
\begin{proof}
 Let $\xymatrix{ \C_{\bullet}  \ar@{^{(}->}[r]^-{i}  & \D_{\bullet}}$ be any cofibration in $\sCat_{\Top}$ (in particular it is an inclusion of categories). The goal is to compute the equalizer of the following diagram:
      $$\xymatrix{
\D_{\bullet}\ar@<1ex>[r] \ar[r] & \D_{\bullet}\bigsqcup_{\C_{\bullet}} \D_{\bullet}
}$$
where the two maps are inclusions of categories coming form the following pushout diagram:
   $$\xymatrix{
 \C_{\bullet} \ar@{^{(}->}[r]^-{i}\ar@{^{(}->}[d]^-{i} & \D_{\bullet}\ar@{^{(}->}[d]^{i_{1}}\\
  \D_{\bullet}\ar@{^{(}->}[r] _{i_{2}}& \D_{\bullet}\bigsqcup_{\C_{\bullet}} \D_{\bullet}
     }$$
     We claim that the equalizer is given exactly by  
     $$\xymatrix{
\C_{\bullet}\ar[r]^-{i} & \D_{\bullet}\ar@<1ex>[r]\ar[r] & \D_{\bullet}\bigsqcup_{\C_{\bullet}} \D_{\bullet}
}$$
First of all, we remark that is a commutative diagram. Suppose that $\C^{'}_{\bullet}$ is an other candidate for the equalizer. Since the functor $\Ob:\sCat\rightarrow\sSet$ commutes with (co)limits ( $\Ob$ admits a left and a right adjoint), There exists a unique map $t$ which makes the following diagram be commutative:
$$\xymatrix{
\Ob\C_{\bullet}^{'}\ar@{.>}[d]^{t} \ar[rd]^{\Ob(F)} & \\
\Ob\C_{\bullet}\ar[r]^-{\Ob(i)}& \Ob\D_{\bullet}\ar@<1ex>[r] \ar[r] & \Ob\D_{\bullet}\bigsqcup_{\Ob\C_{\bullet}} \Ob\D_{\bullet}
}$$
Indeed, the cofibrations in $\sCat_{\Top}$ are injective at the level of objects \ref{topcof3}, and $\sSet$ is cellular \cite{Hirs}.
Now, let $\gamma$ be a morphism in $\C^{'}_{\bullet}$ such that  $i_{1}F(\gamma)=i_{2}F(\gamma)$. Since $i_{1}:\C_{\bullet}\rightarrow\D_{\bullet}\bigsqcup_{\C_{\bullet}}\D_{\bullet}$ and $i_{2}:\C_{\bullet}\rightarrow\D_{\bullet}\bigsqcup_{\C_{\bullet}}\D_{\bullet}$ are inclusions of categories, this implies that  $F(\gamma)$ is a morphism in  $\C_{\bullet}$.
We conclude that any morphism $F:\C^{'}_{\bullet}\rightarrow \D_{\bullet}$ in $\sCat_{\Top}$ such that $i_{1}F=i_{2}F$ is uniquely  factored as a composition: 
$$ \C^{'}_{\bullet}\rightarrow \C_{\bullet}\rightarrow \D_{\bullet}.$$
\end{proof}
\begin{corollary}\label{cellulairescatop}
The model category $\sCat_{\Top}$ is cellular.
\end{corollary}


\section{Model structure on the category of $\infty$-groupoids}\label{section4}
In this section we will prove the existence of a natural cofibrantly model structure on the categories of 
$\infty$-groupoids. 
\begin{definition}\label{groupoid}
Let $\C$ be a topological category, we will say that $\C$ is an $\infty$-groupoid if $\pi_{0}\C$ (the associated homotopy category) is a groupoid.
\end{definition}\label{defgroup}
For every topological category $\D$ we can associate its underlying $\infty$-groupoid $G\D$ by the following pullback diagram:
$$\xymatrix{
   G\D=\iso~\pi_{0}\C\times_{\pi_{0}\D} \D \ar[r]\ar[d]  & \D \ar[d] \\
    \iso~\pi_{0}\D \ar[r]& \pi_{0}\D.
  }
  $$
  \begin{notation}
The category of small $\infty$-groupoids will be denoted by $\infty-\mathbf{Grp}$.

\end{notation}

\begin{lemma}\label{fundamental2}
Let $f: \C\rightarrow \D$ be a map of $\infty-groupoids$, then $f$ is a Dweyer-Kan equivalence of topological categories \cite{Amrani1} if and only if $\Psi f$ is a weak equivalence in $\sSet_{\mathbf{K}}$.
\end{lemma}
\begin{proof}
Suppose that $f$ is a Dwyer-Kan equivalence. We know that $\Psi$ is a right Quillen functor and all objects in $\Cat_{\Top}$ are fibrant. We conclude that $\Psi f$ is a weak equivalence in  $\sSet_{\mathbf{K}}$.
Conversely, suppose that $\Psi f$ ($k^{!}\widetilde{\N}\sing f$) is a weak equivalence in 
$\sSet_{\mathrm{K}}$, we remark that $\Psi \C,  \widetilde{\N}\sing\C,~\Psi \D$ and  $\widetilde{\N}\sing\D$ are Kan complexes since $\C$ and $\D$ are $\infty$-groupoids [\cite{Amrani1}, section 6]. We have the following commutative diagram of weak equivalences [\cite{Amrani1}, section 6]:
$$\xymatrix{
\Psi\C \ar[r]^{\sim} \ar[d]^{\sim} & \Psi\D \ar[d]^{\sim}  \\
    J \widetilde{\N}\sing\C \ar[r]^{\sim} \ar[d]^{id}  & J \widetilde{\N}\sing\D\ar[d]^{id} .\\
    \widetilde{\N}\sing\C\ar[r]^{\sim}  &  \widetilde{\N}\sing\D.
     }
  $$
where $J$ is the Joyal endofunctor on $\sSet$ (more precisely the subcategory of quasi-categories) \cite{Joyalcrm} which associate to each $\infty$-category the biggest Kan sub complex. Moreover the maps  $\Psi \C\rightarrow  \widetilde{\N}\sing\C$ and  $\Psi \D\rightarrow  \widetilde{\N}\sing\D$ are trivial fibrations in $\sSet_{\mathbf{K}}$. 
 But $\sSet_{\mathbf{K}}$ is a left Bousfield localization [\cite{Joyalcrm}, proposition 6.15 ] of $\sSet_{\mathbf{Q}}$, it means that 
$ \widetilde{\N}\sing\C\rightarrow \widetilde{\N}\sing\D$ is an equivalence of $\infty$-categories and so we conclude that $\sing\C\rightarrow \sing\D$ is a Dwyer-Kan equivalence of simplicial categories, consequently $\C\rightarrow \D$ is a Dwyer-Kan equivalence of topological categories.

\end{proof}

\begin{theorem}[\textbf{D}]\label{modelgroupoid}
The adjunction  $(\Theta, \Psi)$ induces a cofibrantly generated model structure on  $\infty-\mathbf{Grp}$, where
\begin{enumerate}
\item a morphism $f:\C\rightarrow \D$ in $\infty-groupoids$  is a weak equivalence (fibration) if 
$$ \Psi f: \Psi\C\rightarrow\Psi\D$$ is a weak equivalence (fibration) in $\sSet_{\mathbf{K}}$,
\item The generating acyclic cofibrations are given by $\Theta\Lambda_{i}^{n}\rightarrow \Theta\Delta^{n}$, for all $0\leq n$ and $0\leq i\leq n$,
\item The generating cofibrations are given by $\Theta\partial\Delta^{n}\rightarrow \Theta\Delta^{n}$, for all $0\leq n$. 
\end{enumerate}
\end{theorem}
\begin{proof}
The category $\infty-\mathbf{Grp}$ is complete by definition and cocomplete because the functor 
$\pi_{0}: \Cat_{\Top}\rightarrow \Cat$ commutes with colimits (has a right adjoint) and the category $\mathbf{Grp}$ (classical groupoids) is cocomplete. Moreover $\Theta$ takes any simplicial set to an $\infty$-groupoid since it commutes with colimits and $\Theta(\Delta^{n})$ is obviously an $\infty$-groupoid. Following  lemma \ref{lem1}, we have to check only the condition 4. Let us take a generating acyclic cofibrantion $\Theta\Lambda_{i}^{n}\rightarrow \Theta\Delta^{n}$, we know that is a Dwyer-Kan equivalence and a cofibration of topological categories since $\Theta$ is a left Quillen functor. If we consider the following pushout in $\infty-\mathbf{Grp}$:

$$\xymatrix{
\Theta\Lambda_{i}^{n} \ar[r] \ar[d]^{\sim}& \C \ar[d]^{f} \\
   \Theta\Delta^{n} \ar[r]  & \D
    }$$
we can deduce that $f$ is a Dwyer-Kan equivalence of topological categories since $\Cat_{\Top}$
has the appropriate model structure \cite{Amrani1}. Finally, we conclude by lemma \ref{fundamental2} 
that $\Psi\C\rightarrow \Psi\D$ is a weak equivalence in $\sSet_{\mathbf{K}}$.
\end{proof}

\begin{remark}
We don't know if the category $\infty-\mathbf{Grp}$ is left proper, but it is right proper for obvious raisons. 
\end{remark}
\begin{theorem}[\textbf{Grothendieck homotopy hypothesis}]\label{equivalence}
The Quillen adjunction  
$$ \xymatrix{\sSet_{\mathbf{K}} \ar@<2pt>[r]^{ \Theta} & \infty-\mathbf{Grp} \ar@<2pt>[l]^{\Psi}}$$ 
induces a Quillen equivalence.
\end{theorem}
\begin{proof}
We should mention the we can't prove the statement directly i.e., that the unit and the counit are equivalences.
First we prove that the functor $\widetilde{\N}\sing: \infty-\mathbf{Grp} \rightarrow \sSet_{\mathbf{K}}$ is well
defined. Let $\C$ be an infinity groupoid then we know \cite{Amrani1} that $\sing\C$ is a simplicial (fibrant) infinity groupoid and that $\widetilde{\N}\sing\C$ is a Kan complex. Consequently the functor  $\widetilde{\N}\sing$ takes Dwyer-Kan equivalences (fibrations) to equivalences (fibrations) in $\sSet_{\mathbf{K}}$ (since $\sSet_{\mathbf{K}}$ is left Bousfield localization of $\sSet_{\mathbf{Q}}$).
So the functor $\widetilde{\N}\sing$ is a well defined right Quillen functor. On the other hand, let $\C$ and $\D$ in $\infty-\mathbf{Grp}$, we have the following commutative diagram of isomorphisms of (derived) mapping spaces in $\Ho(\sSet_{\mathbf{K}})$:

$$\xymatrix{
\map_{\Cat_{\Top}}(\C,\D) \ar[rr]^-{\sim} \ar[d]^{=} &&  \map_{\sSet_{\mathbf{Q}}}(\widetilde{\N}\sing\C,\widetilde{\N}\sing\D)\ar[d]^{=} \\
  \map_{\infty-\mathbf{Grp}}(\C,\D)\ar[rr]^{f} & & \map_{\sSet_{\mathbf{K}}}(\widetilde{\N}\sing\C,\widetilde{\N}\sing\D)
    }$$
The first isomorphism $$\map_{\Cat_{\Top}}(\C,\D) \rightarrow \map_{\sSet_{\mathbf{Q}}}(\widetilde{\N}\sing\C,\widetilde{\N}\sing\D)$$ comes from the fact that $\widetilde{\N}\sing$ is a Quillen equivalence \cite{bergner}, \cite{lurie}, \cite{Amrani1}. The first equality  
$\map_{\sSet_{\mathbf{K}}}(\widetilde{\N}\sing\C,\widetilde{\N}\sing\D)=\map_{\sSet_{\mathbf{Q}}}(\widetilde{\N}\sing\C,\widetilde{\N}\sing\D)$ is a consequence of the fact that $\widetilde{\N}\sing\D$ is a Kan complex. The second equality \\ $\map_{\Cat_{\Top}}(\C,\D)=\map_{\infty-\mathbf{Grp}}(\C,\D)$ is a consequence of the fact that the model full subcategory $\infty-\mathbf{Grp}$ of $\Cat_{\Top}$ has the same weak equivalences (Dwyer-Kan equivalences \ref{fundamental2}) and $\C$ and $\D$ are infinity groupoids.
We conclude that 
$$ \widetilde{\N}\sing: \Ho(\infty-\mathbf{Grp})\rightarrow\Ho(\sSet_{\mathbf{K}}) $$ 
is fully faithful. Now we prove that $\widetilde{\N}\sing$ is essentially surjective. Recall from 
\cite{Joyalcrm} that for any simplicial set $X$ the the natural transformation $\nu_{X}:X\rightarrow k_{!}X$
is a weak equivalence in  $\sSet_{\mathbf{K}}$, so that the map:
$$X\rightarrow k_{!}(X)\rightarrow\widetilde{\N}\sing |\Xi(k_{!}(X)|$$ is a weak equivalence in  $\sSet_{\mathbf{K}}$ since the second map is the unit map of the adjunction (Quillen equivalence) between
$\Cat_{\Top}$ and $\sSet_{\mathbf{Q}}$ which is a weak equivalence of quasi-categories and so a weak equivalence in $\sSet_{\mathbf{K}}$. But $|\Xi(k_{!}(X)|$ is an infinity groupoid. We conclude that 
$\widetilde{\N}\sing$ is essentially surjective. On an other hand, for any infinity groupoid $\C$ we have that
$k^{!}\widetilde{\N}\sing\C\rightarrow J\widetilde{\N}\sing\C=\widetilde{\N}\sing\C$ is a trivial fibration 
\cite{Joyalcrm},\cite{Amrani1}. Consequently, the functor 
 $$k^{!}\widetilde{\N}\sing: \Ho(\infty-\mathbf{Grp})\rightarrow\Ho(\sSet_{\mathbf{K}}) $$ 
 is an equivalence of homotopical (ordinary) categories and its left adjoint is exactly $|\Xi(k_{!}(-)|$. Finally, we conclude that the adjunction $(\Theta, \Psi)$ induces a Quillen equivalence between 
 $\infty-\mathbf{Grp}$  and $\sSet_{\mathbf{K}}$.
\end{proof}
\begin{remark}
The diagonal model structure on $\sCat_{\Top}$ can be \textit{restricted} to a diagonal model structure on $[\Delta^{op}, \infty-\mathbf{Grp}] $. We are pretty sure that this new model structure is also equivalent to $\sSet_{\mathbf{K}}$. Moreover, it is cellular and left proper.
\end{remark}
\subsection{$\mathbf{n-Groupoids}$}\label{section5}
It is well known that any connected topological space $X$ is (zigzag) equivalent to $\mathcal{B}Y$ where 
Y is a topological monoid group like equivalent to $\Omega X$. 
We explain the same result using homotopy hypothesis i.e., every  topological space is zig-zag equivalent to a topological space of the from 
$$\bigsqcup_{x\in\pi_{0}(X)}\mathcal{B}A_{x}$$ 
where $A_{x}$ is a homotopical group (strict multiplication and the inverses are defined up to homotopy) and $\mathcal{B}$ is the bar construction. 

In order to explain this phenomenon, we should recall the interpretation of the coherent nerve $\widetilde{\N}\sing$ for a topological groupoid $\C$. For simplicity we take $\C$ with one object $x$ and suppose that $\mathrm{End}_{\C}(x,x)$ is a topological group
(in general it is a homotopical group), then the geometric realization of $\widetilde{\N}\sing\C$ is nothing else than a model for $\mathcal{B}\mathrm{End}_{\C}(x,x)$  the Bar construction of $\mathrm{End}_{\C}(x,x)$ i.e.,
$$|\widetilde{\N}\sing\C|\sim \mathcal{B}\mathrm{End}_{\C}(x,x).$$
In general, if $X$ is a topological space we associate the $\infty$-groupoid $\mathbf{G}(X)$  given by the formula  \ref{modelgroupoid}
$$X\mapsto\mathbf{G}(X)=|\Xi~k_{!}~\sing(X)|.$$
By Grothendieck homotopy hypothesis theorem \ref{equivalence}, we know that the unit map is an equivalence and the map $\sing(X)\rightarrow k_{!}~\sing(X)$ is also an equivalence \cite{Joyalcrm}
$$\xymatrix{\sing(X)\ar[r]^-{\sim} & k_{!}~\sing(X)\ar[r]^-{\sim} & \widetilde{\N}\sing \mathbf{G}(X)}$$
is an equivalence. On the other hand, the topological realization of the  coherent nerve  $\widetilde{\N}\sing \mathbf{G}(X)$ is equivalent 
to $$\bigsqcup_{x\in [\mathbf{G}(X)] }\mathcal{B}~\mathrm{End}_{ \mathbf{G}(X)}(x,x)$$ 
where $[\mathbf{G}(X)]$ is  the set of chosen objects $x$ of $ \mathbf{G}(X)$, one in each connected component of $\mathbf{G}(X)$.
finally, we end-up with the following zig-zag equivalence 
$$\xymatrix{X & |\sing(X)|\ar[l]_-{\sim} \ar[r]^-{\sim} & |\widetilde{\N}\sing \mathbf{G}(X) |& \bigsqcup_{x\in[\mathbf{G}(X)] }\mathcal{B}~\mathrm{End}_{ \mathbf{G}(X)}(x,x)\ar[l]_-{\sim} .}$$
\begin{definition}
The category $\mathbf{n-Type}$ is the full subcategory of $\Top$ consisting of spaces with the property that all homotopy groups greater than $n$ are vanishing.
We say that  an $\infty$-groupoid is an $n$-groupoid if it is enriched over topological spaces of type $n-1$. We denote the category of $n$-groupoids by $\mathbf{n-Grp}$. 
\end{definition}
\begin{remark}
We conclude that the homotopy category $\Ho(\mathbf{n-Type})\subset\Ho(\Top)$ of spaces of type $n$
is equivalent to the homotopy category $\Ho(\mathbf{n-Grp})\subset\Ho(\infty-\mathbf{Grp})$ of $n$-groupoids.  
\end{remark}
\bibliographystyle{plain} 
\bibliography{diag}

\end{document}